\documentclass[a4paper,12pt]{amsart}

\usepackage{graphicx}

\newtheorem{theorem}{Theorem}[section]
\newtheorem{proposition}[theorem]{Proposition}
\newtheorem{corollary}[theorem]{Corollary}
\newtheorem{lemma}[theorem]{Lemma}
\newtheorem{conjecture}[theorem]{Conjecture}

\newtheorem{hypothesis}[theorem]{Hypothesis}

\newtheorem{Example}[theorem]{Example}

\newtheorem{Remark}[theorem]{Remark}
\newenvironment{remark}{\begin{Remark}\rm}{\end{Remark}}
\newtheorem{Remarks}[theorem]{Remarks}

\newtheorem{Question}[theorem]{Question}

\newcommand{\al}{\alpha}
\newcommand{\be}{\beta}
\newcommand{\ga}{\gamma}
\newcommand{\Ga}{\Gamma}
\newcommand{\de}{\delta}
\newcommand{\eps}{\varepsilon}

\newcommand{\la}{\lambda}
\newcommand{\La}{\Lambda}
\newcommand{\om}{\omega}
\newcommand{\Om}{\Omega}
\newcommand{\si}{\sigma}
\newcommand{\Si}{\Sigma}
\newcommand{\ze}{\zeta}

\let\cal=\mathcal
\let\Bbb=\mathbb

\begin{document}

\title[Prime pairs 15.05.2007]{Prime pairs and zeta's
zeros}


\subjclass[2000]{Primary: 11P32; Secondary:
11M26}

\date{May 15, 2007}

\author{Jacob Korevaar}
\thanks{The author wishes to thank Fokko van de Bult for
numerical support, Jan van de Craats for the drawings
and G\'{e}rald Tenenbaum for useful comment.}

\begin{abstract}
There is extensive numerical support for the
prime-pair conjecture (PPC) of Hardy and
Littlewood (1923) on the asymptotic behavior of
$\pi_{2r}(x)$, the number of prime pairs $(p,\,p+2r)$ with
$p\le x$. However, it is still not known whether there are
infinitely many prime pairs with given even difference!
Using a strong hypothesis on (weighted) equidistribution of
primes in arithmetic progressions, Goldston, Pintz and
Yildirim have recently shown that there are infinitely many
pairs of primes differing by at most sixteen. The present
author uses a Tauberian approach to derive that the PPC is
equivalent to specific boundary behavior of certain
functions involving zeta's complex zeros. Under Riemann's
Hypothesis (RH) and on the real axis these functions
resemble pair-correlation expressions. A speculative 
extension of Montgomery's classical work (1973) would imply 
that there must be an abundance of prime pairs.
\end{abstract}

\maketitle

\setcounter{equation}{0}    
\section{Introduction} \label{sec:1}
As of today, it is not known whether there are
infinitely many prime twins $(p,\,p+2)$, or prime pairs
$(p,\,p+2r)$ with given $r>0$. However, using a 
hypothesis on (weighted) equidistribution of primes in
arithmetic progressions, Goldston, Pintz and Yildirim
\cite{GPY05} have recently shown that there are infinitely
many pairs of primes differing by at most sixteen; see also
Goldston, Motohashi, Pintz and Yildirim \cite{GM06} and the
exposition by Soundararajan \cite{So07}. Let
$$\pi_{2r}(x) =\{\#\,\mbox{prime
pairs}\;(p,\,p+2r)\;\mbox{with}\; p\le x\}.$$ 
Around 1920 Viggo Brun used what is now called Brun's
sieve to prove that $\pi_2(x)=\cal{O}(x/\log^2 x)$.
In 1923 Hardy and Littlewood published a
long paper \cite{HL23} on the Goldbach problems
and on prime pairs, prime triplets, etc. For prime pairs
they conjectured the asymptotic formula
\begin{equation} \label{eq:1.1}
\pi_{2r}(x)\sim 2C_{2r}{\rm
li}_2(x)=2C_{2r}\int_2^x\frac{dt}{\log^2 t}\sim
2C_{2r}\frac{x}{\log^2 x}
\end{equation}
as $x\to \infty$. Here 
\begin{equation} \label{eq:1.2}
C_2 = \prod_{p\,{\rm
prime},\,p>2}\,\left\{1-\frac{1}{(p-1)^2}\right\}
\approx 0.6601618,
\end{equation}
and
\begin{equation} \label{eq:1.3}
C_{2r} = C_2\prod_{p|r,\,p>2}\,\frac{p-1}{p-2}.
\end{equation}
Thus, for example, $C_4=C_8=C_2$, $C_6=2C_2$,
$C_{10}=(4/3)C_2$. There is a great deal of numerical
support for the prime-pair conjecture (PPC). On the Internet
one finds counts  of prime twins for $p$ up to $5\cdot
10^{15}$ by T.\ R.\ Nicely \cite{Ni05}. In Amsterdam Fokko
van de Bult \cite{Bu07} has recently counted the prime pairs
$(p,\,p+2r)$ with $2r\le 10^3$ and $p\le
x=10^3,\,10^4,\,\cdots,\,10^8$. Table $1$ is based on
his work. The bottom line shows (rounded) values $L_2(x)$
of the comparison function $2C_2{\rm li}_2(x)$. The table
supports the conjecture that for every $r$ and $\eps>0$
\begin{equation} \label{eq:1.4}
\pi_{2r}(x)-2C_{2r}{\rm li}_2(x)\ll x^{(1/2)+\eps}.
\end{equation}
Here the symbol $\ll$ is shorthand for the
$\cal{O}$-notation.

\begin{table} \label{table:1}
\begin{tabular}{llllllll} 
$2r\backslash x$ & $10^3$ & $10^4$ & $10^5$ & $10^6$ & $10^7$
& $10^8$ & $C_{2r}/C_2$ \\
      &     &       &        &        &        &         &     \\
2     & 35  & 205   & 1224   & 8169   & 58980  & 440312  & 1   \\ 
4     & 41  & 203   & 1216   & 8144   & 58622  & 440258  & 1   \\ 
6     & 74  & 411   & 2447   & 16386  & 117207 & 879908  & 2   \\
8     & 38  & 208   & 1260   & 8242   & 59595  & 439908  & 1   \\
10    & 51  & 270   & 1624   & 10934  & 78211  & 586811  & 4/3 \\
12    & 70  & 404   & 2421   & 16378  & 117486 & 880196  & 2   \\ 
14    & 48  & 245   & 1488   & 9878   & 70463  & 528095  & 6/5 \\
16    & 39  & 200   & 1233   & 8210   & 58606  & 441055  & 1   \\
18    & 74  & 417   & 2477   & 16451  & 117463 & 880444  & 2   \\
20    & 48  & 269   & 1645   & 10972  & 78218  & 586267  & 4/3 \\
22    & 41  & 226   & 1351   & 9171   & 65320  & 489085  & 10/9\\
24    & 79  & 404   & 2475   & 16343  & 117342 & 880927  & 2   \\ 
30    & 99  & 536   & 3329   & 21990  & 156517 & 1173934 & 8/3 \\
210   & 107 & 641   & 3928   & 26178  & 187731 & 1409150 & 16/5\\
      &     &       &        &        &        &         &     \\
$L_2(x)\,$: & 46 & 214 & 1249 & 8248   & 58754  & 440368 &     \\
       &     &       &        &        &        &        &
\end{tabular}
\caption{Counting prime pairs}
\end{table} 

Sieve methods have become an important part of prime-number
theory. Using an advanced sieve, Jie Wu \cite{Wu04} has
shown that $\pi_2(x)<6.8\,C_2\,x/\log^2x$ for all
sufficiently large $x$. The best result in  the other
direction is  J.\ R.\ Chen's \cite{Che73}: if $N(x)$ denotes the
number of primes $p\le x$ for which $p+2$ has at most two
prime factors, then $N(x)\ge cx/\log^2 x$ for some $c>0$.
There are related results for prime pairs $(p,\,p+2r)$. In
particular, for every $\eps>0$ there is an $x_0=x_0(\eps)$
independent of $r$ such that
\begin{equation} \label{eq:1.5}
\pi_{2r}(x)\le(8+\eps)C_{2r}\,x/\log^2x\quad\mbox{for
all}\;\;x\ge x_0;
\end{equation}
see the book Sieve Methods by Halberstam and Richert
\cite{HR74}. We will also use the fact that the prime-pair
constants $C_{2r}$ have {\it mean value one}, for which
Tenenbaum \cite{Te06} has proposed an elegant proof. There is
a strong estimate in the work of Bombieri and Davenport
\cite{BD66}, which was sharpened by Friedlander and Goldston
\cite{FG95} to
\begin{equation} \label{eq:1.6} 
S_m=\sum_{r=1}^m C_{2r}=m-(1/2)\log m+\cal{O}\{\log^{2/3}
(m+1)\}.
\end{equation}

Starting with Montgomery's work \cite{Mo73} one has
realized that there is a deep connection between the
prime-pair conjectures and the fine distribution of the
complex zeros of the zeta function. Goldston in California 
has been an important contributor to the subject, cf.\
\cite{GM87}, \cite{GGOS00}; several papers exploit the PPC
to obtain plausible results on zeta's zeros. 
Following a lead of Arenstorf \cite{Ar04} we will use a
Wiener--Ikehara theorem to study prime pairs; the two-way
form below is due to the author \cite{Ko05}.
\begin{theorem} \label{the:1.1}
Let $\sum_{n=1}^\infty a_n/n^{w}$ with $a_n\ge 0$ converge
to a sum function $f(w)$ for $w=u+iv$ with $u>1$. Then
\begin{equation} \label{eq:1.7}
\sum_{n\le x}\,a_n\,\sim\, Ax\quad\mbox{as}\;\;x\to\infty
\end{equation} 
if and only if for $u\searrow 1$, the difference
\begin{equation} \label{eq:1.8}
f(u+iv)-\frac{A}{u+iv}=g(u+iv)
\end{equation}
has a distributional limit $g(1+iv)$, which on every finite
interval $(-B,B)$ coincides with a pseudofunction (that
may {\rm a priori} depend on $B$). 
\end{theorem}
Ikehara \cite{Ik31} and Wiener \cite{Wi32} obtained
(\ref{eq:1.7}) under the hypothesis that $g(w)$ has an
analytic or continuous extension to the half-plane $\{u\ge
1\}$. The condition $\sum_{n\le x}\,a_n=\cal{O}(x)$ would
ensure that $f(u+iv)$ and $g(u+iv)$ have a distributional
limit as $u\searrow 1$. A pseudofunction is the
distributional Fourier transform of a bounded function
which tends to zero at $\pm\infty$; locally, such a
distribution is given by trigonometric series with
coefficients that tend to zero. A pseudofunction cannot
have pole-type singularities. In the case $a_n\ge 0$, local
pseudofunction boundary behavior of $g(w)$ in (\ref{eq:1.8})
implies that
\begin{equation} \label{eq:1.9}
(w-w_0)g(w)\to 0
\end{equation} 
for angular approach of $w$ (from the right) to any point
$w_0$ on the line $\{u=1\}$; cf.\ \cite{Ko02},
or \cite{Ko04}, Theorem III.3.1.

\setcounter{equation}{0}    
\section{Basic auxiliary functions}
\label{sec:2}
For analytic formulation of the general PPC one may
introduce the sums
$$\theta_{2r}(x)=\sum_{p,\,p+2r\,{\rm prime};\;p\le
x}\,\log^2 p.$$
Relation (\ref{eq:1.1}) is equivalent to the asymptotic
formula
$$\theta_{2r}(x)\sim 2C_{2r}x\quad\mbox{as}\;\;x\to\infty.$$ 
By the Wiener--Ikehara theorem this relation holds
if and only if the function
$$\tilde D_{2r}(w)=\sum_{p,\,p+2r\,{\rm prime}}\,\frac{\log^2
p}{p^w}$$
can be written as $2C_{2r}/(w-1)+g_{2r}(w)$, where 
$g_{2r}(w)$ has `good boundary behavior' as $u\searrow 1$.

At this stage it is convenient to replace
$\theta_{2r}(x)$ and $\tilde D_{2r}(w)$ by functions with
similar behavior that involve von Mangoldt's function
$\La(n)$. We recall its generating Dirichlet series; using
the Euler product for $\ze(w)$, 
$$\sum_{n=1}^\infty\,\frac{\La(n)}{n^w}=-\frac{\ze'(w)}{\ze(w)}
=\sum_{p\,{\rm prime}}\,(\log p)
\bigg(\frac{1}{p^w}+\frac{1}{p^{2w}}+\cdots\bigg).$$
One has $\La(k)=\log p$ if
$k=p^\al$ with $p$ prime, and $\La(k)=0$ if $k$ is not a
prime power. Since there are only $\cal{O}(\sqrt x)$ prime
powers $p^\al\le x$ with $\al\ge 2$, the difference between
\begin{equation} \label{eq:2.1}
\psi_{2r}(x)\stackrel{\mathrm{def}}{=}
\sum_{n\le x}\,\La(n)\La(n+2r)
\end{equation}
and $\theta_{2r}(x)$ is not much larger than $\sqrt x$.
Thus the PPC is also equivalent to the relation
\begin{equation} \label{eq:2.2}
\psi_{2r}(x)\sim 2C_{2r}x\quad\mbox{as}\;\;x\to\infty.
\end{equation}
Similarly, the function 
\begin{equation} \label{eq:2.3}
D_{2r}(s)\stackrel{\mathrm{def}}{=}\sum_{n=1}^\infty
\,\frac{\La(n)\La(n+2r)}{n^s(n+2r)^s}\qquad(s=\si+i\tau,\,
\si>1/2)
\end{equation}
behaves in the same way as $\tilde D_{2r}(2s)$ when $2\si$
is close to $1$. Setting
\begin{equation} \label{eq:2.4}
D_{2r}(s)-\frac{C_{2r}}{s-1/2}=G_{2r}(s),
\end{equation}
the Wiener--Ikehara theorem with $2s$ instead of $w$
shows that the PPC (\ref{eq:2.2}) is equivalent to good
boundary behavior of $G_{2r}(s)$ as
$\si\searrow 1/2$. 

\medskip\noindent {\scshape Combinations}.
In order to profit from the fact that the constants $C_{2r}$
have mean value $1$ it helps to study sums 
$\sum_{2r\le\la}\,D_{2r}(s)$ for large values of $\la$.
Indeed, under the PPC, their boundary behavior should be
roughly like that of $(\la/2)/(s-1/2)$. In this spirit we
will study manageable combinations $V^\la(s)$ of functions
$D_{2r}(s)$ with nonnegative coefficients. They are derived
from a certain repeated complex integral $T^\la(s)$ (see
Section \ref{sec:5}) which extends and modifies an integral
of Arenstorf \cite{Ar04}. It involves a parameter
$\la>0$ and a parameter function $E^\la$; the
resulting formula for $V^\la(s)$ is
\begin{equation} \label{eq:2.5}
V^\la(s) \stackrel{\mathrm{def}}{=} 2\sum_{0<2r\le\la}
\,E^\la(2r)D_{2r}(s) = T^\la(s) - D_0(s) + H^\la(s).
\end{equation} 
Here the function $D_{2r}(s)$ is given by (\ref{eq:2.3}),
also when $r=0$, and $H^\la(s)$ is holomorphic for
$\si>0$. The parametric function
$E^\la(\nu)=E(\nu/\la)$ acts as a sieving device. The basic 
function $E(\nu)$ is taken even, with compact support,
Lipschitz continuous and decreasing on $[0,\infty)$. For
convenience $E(\nu)$ is normalized so that its support is
$[-1,1]$ and $E(0)=1$. The simplest sieving function
$E^\la(\nu)$ is given by the Fourier transform of the
Fej\'{e}r kernel for ${\Bbb R}$,
$$E^\la_F(\nu) = \frac{1}{\pi}\int_0^\infty\frac{\sin^2(\la
t/2)}{\la(t/2)^2}\cos \nu t\,dt =
\left\{\begin{array}{ll} 1-|\nu|/\la & \mbox{for
$|\nu|\le\la$,}\\  0 & \mbox{for $|\nu|\ge\la$.}
\end{array}\right.$$
This function is adequate if one is willing to use Riemann's
Hypothesis (RH) in the proof of the main theorem; cf.\ the
manuscript \cite{Ko07}. In the present paper we will prove
the main result without appealing to RH, but for that have to
require that $E$ be sufficiently {\it smooth}. More
precisely, we suppose that $E$, $E'$ and $E''$ are absolutely
continuous with $E'''$ of bounded variation. One could for
example use the Fourier transform of the Jackson kernel
for ${\Bbb R}$,
\begin{align} 
E^\la_J(\nu) &= 
\frac{3}{4\pi}\int_0^\infty\frac{\sin^4(\la t/4)}
{\la^3(t/4)^4}\cos \nu t\,dt
\notag \\ &=\left\{\begin{array}{ll}
1-6(\nu/\la)^2+6(|\nu|/\la)^3 & \mbox{for
$|\nu|\le\la/2$},\\ 2(1-|\nu|/\la)^3 & \mbox{for
$\la/2\le|\nu|\le\la$},\\ 0 & \mbox{for $|\nu|\ge\la$.}
\end{array}\right.\notag
\end{align} 

The PPC and the mean value $1$ of the constants $C_{2r}$ lead
one to expect that for large $\la$, $V^\la(s)$ has a
first-order pole at $s=1/2$ with residue
\begin{equation} \label{eq:2.6}
2\sum_{0<2r\le\la}\,E(2r/\la)C_{2r}\approx\la\int_0^1
E(\nu)d\nu \stackrel{\mathrm{def}}{=} A^E\la.
\end{equation}

For the following we need a Mellin transform associated
with the Fourier transform of the kernel $E^\la$:
\begin{equation} \label{eq:2.7}
M^\la(z)=M^\la_E(z) \stackrel{\mathrm{def}}{=}
\frac{1}{\pi}\int_0^\infty \hat E^\la(t)t^{-z}dt,\quad
-1<x={\rm Re}\,z<1.
\end{equation}
\begin{proposition} \label{prop:2.1}
For our smooth $E$, the Mellin transform has a meromorphic
extension to the half-plane $\{x>-3\}$, given by
\begin{equation} \label{eq:2.8}
M^\la(z) = \frac{2}{\pi}\la^z\Ga(-z-3)\sin(\pi z/2)
\int_0^{1+} \nu^{z+3}dE'''(\nu)=\la^zM(z),
\end{equation}
say. It has poles (of the first order) at $z=1,\,3,\,\cdots$.
The residue at $z=1$ is $-(2\la/\pi)A^E=-(2\la/\pi)\int_0^1
E(\nu)d\nu$ and $M^\la(0)=1$. For fixed $\la$ and any
constant $C$ one has the majorization
\begin{equation} \label{eq:2.9}
M^\la(x+iy)\ll(|y|+1)^{-x-7/2}\quad\mbox{for}\;\;-3<x\le
C,\;\; |y|\ge 1.
\end{equation}
\end{proposition}
\begin{proof}
The Fourier transform $\hat E^\la(t)$ is
$\cal{O}\{(|t|+1)^{-2}\}$. By (\ref{eq:2.7}), initially
taking $0<x<1$ and using the Mellin transform of $\cos\nu t$,
cf.\ Section \ref{sec:4},
\begin{align} 
M^\la(z) &= \frac{1}{\pi}\int_0^\infty
t^{-z}dt\cdot 2\int_0^\la E^\la(\nu)(\cos t\nu)d\nu
\notag \\ & =\frac{2}{\pi}\int_0^\la E^\la(\nu)d\nu
\int_0^{\infty-}(\cos\nu t)t^{-z}dt\notag\\ &=
\frac{2}{\pi}\Ga(1-z)\cos\{\pi(1-z)/2\}\int_0^\la
E^\la(\nu)\nu^{z-1}d\nu \notag \\ &=
\frac{2}{\pi}\la^z\Ga(1-z)\sin(\pi z/2)\int_0^1
E(\nu)\nu^{z-1}d\nu.\notag
\end{align}
For $x>0$ and smooth $E$, the final integral may also be
written as
\begin{align}
-\frac{1}{z}\int_0^1 \nu^zdE(\nu) &=
\frac{1}{z(z+1)}\int_0^{1}\nu^{z+1}dE'(\nu) \notag \\ &=
\frac{1}{z(z+1)(z+2)(z+3)}\int_0^{1+}\nu^{z+3}dE'''(\nu).
\notag
\end{align}
This is enough to prove (\ref{eq:2.8}), hence $M^\la(z)$
has a  meromorphic extension to the half-plane $\{x>-3\}$.
The poles  of $\Ga(-z-3)$ at $z=-2,\,0,\,2,\,\cdots$ are
cancelled by zeros of $\sin(\pi z/2)$ and the pole at $z=-1$
is cancelled by the zero of $\int_0^1\nu^{z+1}dE'(\nu)$ at
that point. The formulas also show that
$M^\la(0)=1$ and that the residue at the pole $z=1$ is equal
to $-(2\la/\pi)\int_0^1 E(\nu)d\nu$. The order estimate
(\ref{eq:2.9}) follows from the standard inequalities
\begin{equation} \label{eq:2.10}
\Ga(z)\ll |y|^{x-1/2}e^{-\pi|y|/2},\quad\sin(\pi z/2)\ll
e^{\pi|y|/2} 
\end{equation}
which are valid for $|x|\le C$ and $|y|\ge 1$. The inequality
for $\Ga(z)$ follows from Stirling's formula for complex
$z$; see formula (\ref{eq:8.3}) below. 
\end{proof}

\setcounter{equation}{0}    
\section{Results}\label{sec:3}
Our results involve the complex zeros $\rho$ of the
zeta function. Taking multiplicities into account, the
zeros above the real axis will be arranged according to
non-decreasing imaginary part:
$$\rho=\rho_n=\be_n+i\ga_n,\;\;0<\ga_1\approx
14<\ga_2\approx 21\le\cdots,\;\;n=1,\,2,\,\cdots$$
(with $\be_n=1/2$ as far as zeros have been computed); we
write $\overline\rho_n=\rho_{-n}$. The theorem below involves
the sum
\begin{align} & \label{eq:3.1}
\Si^\la(s) \stackrel{\mathrm{def}}{=} 
\left\{\frac{\ze'(s)}{\ze(s)}\right\}^2 \notag \\ & \quad
+2\,\frac{\ze'(s)}{\ze(s)}\,
\sum_{\rho}\,\Ga(\rho-s)M^\la(\rho-s)
\cos\{\pi(\rho-s)/2\} \\ & 
+ \sum_{\rho,\,\rho'}\,\Ga(\rho-s)\Ga(\rho'-s) 
M^\la(\rho+\rho'-2s)\cos\{\pi(\rho-\rho')/2\},\notag
\end{align}
where $M^\la(\cdot)$ is given by (\ref{eq:2.7}). It is 
convenient to denote the sum of the first two terms by 
$\Si^\la_1(s)$ and to set the double sum equal to
$\Si^\la_2(s)$. Results from Section \ref{sec:2} show that
$\Si^\la_1(s)$ defines a meromorphic function for $\si<3$
whose only poles in the strip $\{0<\si<1\}$ occur at the
complex zeros of $\ze(\cdot)$. Under RH the double
series, in which $\rho$ and $\rho'$ both run over zeta's
complex zeros, is absolutely convergent for $1/2<\si<1$;
cf.\ Lemma \ref{lem:4.2} below. Without RH the double sum may
be interpreted as a limit of sums over the zeros $\rho$,
$\rho'$ with $|{\rm Im}\,\rho|$,
$|{\rm Im}\,\rho'|<R$; it will follow from Theorem
\ref{the:3.1} that the combination $\Si^\la(s)$ is in any
case holomorphic for $1/2<\si<1$.

The formula for $V^\la(s)$ in (\ref{eq:2.5}) contains the
function $T^\la(s)$ for which a repeated complex integral is
introduced in Section \ref{sec:5}. Moving the paths of
integration in this integral and using the residue theorem
one obtains
\begin{theorem} \label{the:3.1}
For any $\la>0$, any smooth sieving function $E$, and for
$s=\si+i\tau$ with $1/2<\si<1$ there are holomorphic
representations 
\begin{align}  \label{eq:3.2}
V^\la(s) &= 2\sum_{0<2r\le\la}\,E(2r/\la)D_{2r}(s) \notag \\
&= \frac{-1/4}{(s-1/2)^2}+\frac{A^E\la}{s-1/2}
+\Si^\la(s)+H^\la(s) \\ &=
\frac{A^E(\la-1)}{s-1/2}+\Si^\la(s)-\Si^1(s)+H^\la(s),
\notag \label{eq:3.2}
\end{align}
where $A^E=\int_0^1 E(\nu)d\nu$ and $\Si^\la(s)$ is given by
$(\ref{eq:3.1})$ (with proper interpretation of the double
sum); the various functions $H^\la(s)$ are analytic for
$1/2\le\si<1$, and for $1/4<\si<1$ under RH. On the interval
$\{1/2\le s\le 3/4\}$ one has $H^\la(s)=\cal{O}(\la\log\la)$
as $\la\to \infty$.
\end{theorem}
The (extended) Wiener--Ikehara Theorem will
now show that the Hardy--Littlewood conjectures for prime
pairs $(p,\,p+2r)$ are true if and only if the differences
$\Si^\la(s)-\Si^1(s)$ exhibit certain specific boundary
behavior  as $\si\searrow 1/2$; cf.\ (\ref{eq:2.4}). To make
this precise, define
\begin{equation} \label{eq:3.3}
R(\la)=2\sum_{0<2r\le\la}
\,E(2r/\la)C_{2r}-A^E(\la-1).
\end{equation}
By induction Theorems \ref{the:3.1} and \ref{the:1.1} imply
\begin{corollary} \label{cor:3.2}
Suppose that the prime-pair conjecture for pairs
$(p,\,p+2r)$ is true for every $r<m$. Then the PPC for
prime pairs $(p,\,p+2m)$ is true if and only if for some
(or every) smooth function $E$ and some (or every)
number $\la\in(2m,2m+2]$, the function
\begin{equation} \label{eq:3.4}
G^\la(s) \stackrel{\mathrm{def}}{=}
\Si^\la(s)-\Si^1(s)-\frac{R(\la)}{s-1/2}
\end{equation} 
has good (local pseudofunction) boundary behavior as
$\si\searrow 1/2$. 
\end{corollary}
The double series in (\ref{eq:3.1}) defines $\Si^\la_2(s)$ 
as a meromorphic function for $1/2<\si<1$ whose poles occur
at complex zeros of $\ze(\cdot)$, and these poles are
cancelled by those of $\Si^\la_1(s)$. Formally there is
cancellation also at the other complex zeros of
$\ze(\cdot)$. Turning to the function
$G^\la(s)$, note that by (\ref{eq:3.2}), it {\it does}
have good boundary behavior when $\la\le 2$; under RH, it
will even be analytic for $1/4<\si<1$. Indeed,
$V^\la(s)=0$ for $\la\le 2$. These observations support the
following 
\begin{conjecture} \label{con:3.3}
For every $\la>0$ and every $E$, the function $G^\la(s)$ in
$(\ref{eq:3.4})$ has an analytic continuation to the strip
$\{1/4<\si<1\}$. 
\end{conjecture}
\noindent If this is true, the counting functions
$\pi_{2r}(x)$ all satisfy estimates of type (\ref{eq:1.4}).

\smallskip\noindent{\scshape Conditional abundance of prime
pairs.} It will follow from Section \ref{sec:7} that the part
of the sum $\Si^\la_2(s)$ in which ${\rm Im}\,\rho$ and
${\rm Im}\,\rho'$ have the same sign defines a meromorphic
function for $1/2\le\si<1$ whose only poles occur at 
complex zeros of $\ze(\cdot)$. Thus for a study of its
pole-type behavior near the point
$s=1/2$, the double sum $\Si^\la_2(s)$ in (\ref{eq:3.1})
may be reduced to the sum $\Si^\la_3(s)$ in which ${\rm
Im}\,\rho$ and ${\rm Im}\,\rho'$ have opposite sign. Hence
in the study of the PPC under RH, the {\it differences} of
zeta's zeros on the same side of the real axis play a key
role.
\begin{theorem} \label{the:3.4}
Assume RH. Then the pole-type behavior of $\Si^\la_3(s)$ and
$\Si^\la(s)$ as $s\searrow 1/2$ is the same as that of the
reduced sum
\begin{equation} \label{eq:3.5}
\Si^\la_4(s) = 2\pi\sum_{\ga,\,\ga';\,|\ga'-\ga|<\ga^{1/2}}\,
\ga^{-2s+i(\ga-\ga')}M^\la\{1-2s+i(\ga-\ga')\},
\end{equation}
where $\ga$ and $\ga'$ run over the imaginary parts of
the zeros of $\ze(\cdot)$ in the upper half-plane. 
\end{theorem}
The expression in (\ref{eq:3.5}) is reminiscent of the
pair-correlation function of zeta's complex zeros which
was studied by Montgomery \cite{Mo73} et al. Since the 
constants $C_{2r}$ have mean value $1$, the function $R(\la)$
in (\ref{eq:3.3}) is $o(\la)$ as $\la\to\infty$; cf.\
(\ref{eq:2.6}). [By (\ref{eq:1.6}) it will even be
$\cal{O}(\log\la)$.] The corresponding hypothesis below
regarding $\Si^\la(s)-\Si^1(s)$ would follow from a plausible
extension of Montgomery's work; see Section \ref{sec:9}.
\begin{hypothesis} \label{hyp:3.5}
For smooth $E$ the `upper residue'
\begin{equation} \label{eq:3.6}
\om(\la)=\om^E(\la)=\limsup_{s\searrow
1/2}\,(s-1/2)\{\Si^\la(s)-\Si^1(s)\}
\end{equation}
is $o(\la)$ as $\la\to\infty$.
\end{hypothesis}
It follows from (\ref{eq:3.2}) and (\ref{eq:1.5}) that
$\om(\la)=\cal{O}(\la)$. If Hypothesis \ref{hyp:3.5} is true
there will be an abundance of prime pairs:
\begin{theorem} \label{the:3.6}
Assume Hypothesis \ref{hyp:3.5}. Then for 
every $\eps>0$, there is a positive integer $m$,
depending on $\om(\cdot)$ and $\eps$, such that 
\begin{equation} \label{eq:3.7}
\limsup_{x\to \infty}\,\frac{1}{m}
\sum_{r\le m}\,\frac{\pi_{2r}(x)}{x/\log^2x}>2-\eps.
\end{equation}
\end{theorem} 
\noindent Here the constant $2$ would be optimal.

We finally mention an interesting positivity property of
certain double sums $\Si^\la_2(s)$ in (\ref{eq:3.1}): 
\begin{proposition} \label{prop:3.7}
Let $E^\la$ be a sieving function (such as
$E^\la_J$) for which $\hat E^\la(t)\ge 0$. Then
$\Si^\la_2(s)\ge 0$ when $1/2<s<1$.
\end{proposition} 
This positivity  and a speculative equidistribution result
for prime pairs with different values of $2r$
would also imply that there is an abundance of prime pairs;
see Section
\ref{sec:10}.

\setcounter{equation}{0}    
\section{Complex representation for $E^\la(\al-\be)$}
\label{sec:4} 
For the discussion of $T^\la(s)$ in Section \ref{sec:5} we
need a complex integral for the sieving function
$E^\la(\al-\be)$ in which $\al$ and $\be$ occur separately.
It is obtained from the representation of $E^\la(\al-\be)$ as
an inverse Fourier (cosine) transform:
$$E^\la(\al-\be)=\frac{1}{\pi}
\int_0^\infty \hat E^\la(t)\cos\{(\al-\be)t\}dt$$ 
and a repeated complex integral for 
$$\cos\{(\al-\be)t\}=\cos\al t\cos\be t +\sin\al t\sin\be
t.$$
To set the stage we start with a complex representation
for $\cos\al$ and $\sin\al$ with $\al>0$. Setting $z=x+iy$
we write $L(c)$ for the `vertical line' $\{x=c\}$; the
factor $1/(2\pi i)$ in complex integrals will be omitted.
Thus
$$\int_{L(c)}f(z)dz\stackrel{\mathrm{def}}{=}
\frac{1}{2\pi i}\int_{c-i\infty}^{c+i\infty}f(z)dz.$$
Mellin inversion of the improper Euler integral  
$$\int_0^{\infty-}(\cos\al)\al^{z-1}d\al=\Ga(z)\cos(\pi
z/2)\qquad(0<x<1)$$
now gives the improper complex integral 
\begin{align}
\cos\al &= \int_{L(c)}^*\Ga(z)\al^{-z}\cos(\pi
z/2)dz \notag \\ &= \lim_{A\to\infty}\frac{1}{2\pi
i}\int_{c-iA}^{c+iA}\,\cdots\quad(0<c<1/2);\notag
\end{align}
there is a similar representation for $\sin\al$.
It is important for us to have absolutely convergent
integrals. We therefore replace the line $L(c)$ by a
path $L(c,B)=L(c_1,c_2,B)$ with suitable $c_1<c_2$ and
$B>0$ (cf.\ Figure \ref{fig:1}):
\begin{equation} \label{eq:4.1}
L(c,B)=\left\{\begin{array}{lllll}
\mbox{$\quad$the half-line} & \mbox{$\{x=c_1,\,-\infty<y\le
-B\}$}\\
\mbox{$+\;$the segment} & \mbox{$\{c_1\le x\le
c_2,\,y=-B\}$}\\
\mbox{$+\;$the segment} & \mbox{$\{x=c_2,\,-B\le y\le B\}$}\\
\mbox{$+\;$the segment} & \mbox{$\{c_2\ge x\ge
c_1,\,y=B\}$}\\
\mbox{$+\;$the half-line} & \mbox{$\{x=c_1,\,B\le
y<\infty\}$.}  
\end{array}\right.
\end{equation}
Taking $c_1<-1/2$, $c_2>0$ and using formula (\ref{eq:2.10}),
one thus obtains the absolutely convergent repeated integral
\begin{align}
\cos\{(\al-\be)t\} &= \int_{L(c,B)}
\Ga(z)\al^{-z}t^{-z}dz\,\cdot \notag \\&  \quad\cdot
\int_{L(c,B)}\Ga(w)\be^{-w}t^{-w}\cos\{\pi(z-w)/2\}dw.
\notag
\end{align}
Multiplying both sides by $\hat E^\la(t)$,
integrating over $\{0<t<\infty\}$, inverting order of
integration and using formula (\ref{eq:2.7}) for
$M^\la(\cdot)$, one obtains
\begin{proposition} \label{prop:4.1}
Let $\al,\,\be>0$, $-3/2<c_1<0<c_2<1/2$ and $B>0$. Then
for smooth sieving functions $E^\la(\cdot)$ as in Section
\ref{sec:2} one has the representation
\begin{align} \label{eq:4.2}
E^\la(\al-\be) & = \int_{L(c,B)}\Ga(z)\al^{-z}dz
\int_{L(c,B)}\Ga(w)\be^{-w}\,\cdot\notag
\\ & \quad\cdot M^\la(z+w)\cos\{\pi(z-w)/2\}\,dw.
\end{align}
\end{proposition}
\begin{figure}[htb]
$$\includegraphics[width=4cm]{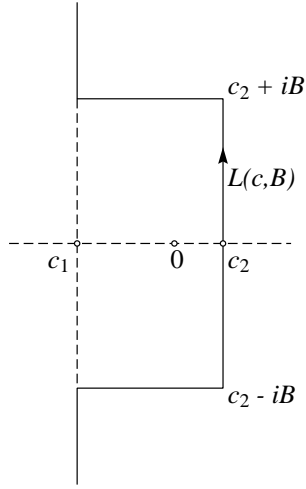}
$$
\caption{The path $L(c_1,c_2,B)$ \label{fig:1}}
\end{figure}
To justify the operations and verify the absolute
convergence of the integral in (\ref{eq:4.2}) one may use
the estimate (\ref{eq:2.9}) and a simple lemma: 
\begin{lemma} \label{lem:4.2}
For real constants $a,\,b,\,c$, the function
$$\phi(y,v)=(|y|+1)^{-a}(|v|+1)^{-b}(|y+v|+1)^{-c}$$
is integrable over ${\Bbb R}^2$ if (and only if) $a+b>1$,  
$a+c>1$, $b+c>1$ and $a+b+c>2$. For integrability over
${\Bbb R}_+^2$ the condition $a+b>1$ may be dropped. 
\end{lemma}
For verification let $A>0$. On the subset of ${\Bbb R}^2$
where $|y|\le A$, the condition $b+c>1$ is necessary and
sufficient for a finite $v$-integral. Similarly for $|v|\le
A$ and the condition $a+c>1$. When $|y+v|\le A$ one needs
the condition $a+b>1$ in the case of ${\Bbb R}^2$. For the
fourth condition one looks at the set where $v\ge y\ge 1$.
Setting $v=yr$ with new variable $r$, 
$$\phi(y,v)dydv\asymp y^{-a}(yr)^{-b}\{y(r+1)\}^{-c}ydydr,$$
and the right-hand side is integrable over the set
$\{1<y<\infty,\,1<r<\infty\}$ if and only if $b+c>1$
and $a+b+c>2$. Similarly when $y$ and $v$ have opposite
sign; one may of course assume that $|y+v|>1$ then. 
 
We turn to Proposition \ref{prop:4.1}. Setting $z=x+iy$,
$w=u+iv$, Stirling's formula and (\ref{eq:2.9}) give the
following majorant for the integrand in (\ref{eq:4.2}) on
the remote parts of the paths $L(c,B)$: 
$$(|y|+1)^{c_1-1/2}(|v|+1)^{c_1-1/2}(|y+v|+1)^{-2c_1-7/2}.$$
For integrability one thus needs $-3<c_1<0$. The more
stringent requirements in the proposition 
serve to justify inversion of the order of integration in a
triple integral and to keep the paths within the strip
$\{-3<X<1\}$ where $M^\la(Z)$ is known to be regular.

\setcounter{equation}{0}    
\section{The complex integral for $T^\la(s)$} \label{sec:5}
The function $T^\la(s)$ in (\ref{eq:2.5}) is
defined by the integral below for $\si>1+|c_1|$, while for
$s$ with smaller real part it is defined by analytic
continuation;
\begin{align}  \label{eq:5.1}
\quad T^\la(s) &=
\int_{L(c,B)}\Ga(z)\frac{\ze'(z+s)}{\ze(z+s)}\,dz
\int_{L(c,B)}\Ga(w)\frac{\ze'(w+s)}{\ze(w+s)}\,
\cdot \notag \\ & \qquad\cdot 
M^\la(z+w)\cos\{\pi(z-w)/2\}dw.
\end{align}
\begin{theorem} \label{the:5.1}
Let $-3/2<c_1<0<c_2<1/2$. Then the integral
$(\ref{eq:5.1})$ defines $T^\la(s)$ as a holomorphic
function of $s=\si+i\tau$ for $\si>1-c_1$. Assuming RH, the
integral gives $T^\la(s)$ as a holomorphic function for
$\si>\max\{(1/2)-c_1,1-c_2\}$ and $|\tau|<B$. 

The integral has an analytic continuation to the half-plane 
$\{\si>1/2\}$ given by the expansion
\begin{equation} \label{eq:5.2}
T^\la(s)=\sum_{k,l}\La(k)\La(l)k^{-s}l^{-s}E^\la(k-l).
\end{equation}
\end{theorem}
\begin{proof}[Discussion]
For $z\in L(c,B)$ and $\si>1-c_1$, the sum
$z+s$ will stay away from the poles of $\ze'/\ze$. Under
RH the same holds when $\si>\max\{(1/2)-c_1,1-c_2\}$
and $|\tau|<B$. Indeed, in that case $x+\si>1/2$ and also
$z+s\ne 1$: if $x+\si=1$, then $z$ must lie on the part of
$L(c,B)$ where $|y|\ge B$, and then $y+\tau\ne 0$. Similarly
for $w=u+iv$. The absolute convergence of the repeated
integral in (\ref{eq:5.1}) can be proved in the same way as
that of (\ref{eq:4.2}). Indeed, the quotient $(\ze'/\ze)(Z)$
grows at most logarithmically in $Y$ for $X\ge 1$, and for
$X\ge (1/2)+\eta$ under RH; cf.\ Titchmarsh \cite{Ti86}. The
holomorphy of the integral for $T^\la(s)$ now follows from
locally uniform convergence in $s$. 

For the second part we substitute the Dirichlet series for
$(\ze'/\ze)(\cdot)$ into (\ref{eq:5.1}),
initially taking $\si>1-c_1$. Integrating term by term and 
applying Proposition \ref{prop:4.1} one obtains the
expansion (\ref{eq:5.2}). Because $E^\la(k-l)\ne 0$
only for finitely many values of $k-l$, the
series represents a holomorphic function for $\si>1/2$;
cf.\ the proof of Theorem \ref{the:5.2} below. The sum of
the series provides an (the) analytic continuation of the
integral to the half-plane $\{\si>1/2\}$. 
\end{proof}
We will now derive (\ref{eq:2.5}).
\begin{theorem} \label{the:5.2}
For arbitrary $\la>0$ and $\si>1/2$,
\begin{equation} \label{eq:5.3}
T^\la(s)=D_0(s)+2\sum_{0<2r\le\la}\,E(2r/\la)D_{2r}(s)
+H^\la_1(s),
\end{equation}
where $H^\la_1(s)$ is holomorphic for $\si>0$. On the real
interval $\{1/2\le s\le 3/4\}$ one has
$H^\la_1(s)=\cal{O}(\la)$ as $\la\to \infty$.
\end{theorem}
\begin{proof}
Taking $k=l$ in (\ref{eq:5.2}) one obtains the term
$D_0(s)$ in (\ref{eq:5.3}). For $|k-l|=2r$ one
obtains a constant multiple of $D_{2r}(s)$. The coefficient
is different from $0$ only if $2r<\la$ and in fact equal
to $2E(2r/\la)$. If
$|k-l|=2r-1$ one can have $\La(k)\La(l)\ne 0$ only if either
$k$ or $l$ is of the form $2^\al$ for some $\al>0$. The
resulting functions, for which $2r-1$ must be $<\la$, are
holomorphic for $\si>0$. For real $s\searrow 1/2$ the sum
$H^\la_1(s)$ of their values will be $\cal{O}(\la\log\la)$
as $\la\to \infty$.
\end{proof}
It remains to determine the analytic character of $D_0(s)$:
\begin{lemma} \label{lem:5.3}
One has
\begin{equation} \label{eq:5.4}
D_0(s)=\sum_{k=1}^\infty\,\frac{\La^2(k)}{k^{2s}}
=\frac{1/4}{(s-1/2)^2}+H_0(s),
\end{equation}
where $H_0(s)$ is analytic for $\si\ge 1/2$, and for
$\si>1/4$ under RH. 
\end{lemma}
Indeed, for $x>1$
\begin{align} &
\sum_k\La^2(k)k^{-z} = \sum_p(\log^2 p)p^{-z}+H_1(z)
= -\frac{d}{dz}\sum_p(\log p)p^{-z}+H_1(z) \notag \\ &=
-\frac{d}{dz}\sum_k\La(k)k^{-z}+H_2(z) =
\frac{d}{dz}\frac{\ze'(z)}{\ze(z)}+H_2(z)=
\frac{1}{(z-1)^2}+H_3(z),\notag
\end{align}
where $H_1(z)$ and $H_2(z)$ define holomorphic functions for
$x>1/2$, while $H_3(z)$ is holomorphic for $x\ge 1$, and
for $x>1/2$ under RH. Finally take $z=2s$.

\setcounter{equation}{0}    
\section{Transformation of the integral for $T^\la(s)$}
\label{sec:6}  
Taking $c_1,\,c_2$ and $s$ as in the first part of
Theorem \ref{the:5.1} we will move the paths of integration
in the integral for $T^\la(s)$, but first change variables.
Replacing $z$ by $z'-s$ and $w$ by $w'-s$ (and subsequently
dropping the primes), one obtains
\begin{align} \label{eq:6.1}
T^\la(s) & = \int_{L(c',B')}\Ga(z-s)
\frac{\ze'(z)}{\ze(z)}\,dz
\int_{L(c',B')}\Ga(w-s)\frac{\ze'(w)}{\ze(w)}\,
\cdot \notag \\ & \quad\; \cdot
\,M^\la(z+w-2s)\cos\{\pi(z-w)/2\}dw.
\end{align}
Here the paths of integration will initially depend on $s$:
$c'_1=c_1+s$, $c'_2=c_2+s$, and the horizontal segments 
may be at different distances from the real axis. However,
by our standard estimates and Cauchy's theorem, one may
fix $c'_1=1$ and $c'_2=3/2$, say, use a constant $B'$
and take $1<\si<3/2$, $|\tau|<B'$. Observe that henceforth,
the point $s$ will be to the left of the paths. 

Starting with (\ref{eq:6.1}), where we rename
$c'_1=1=c_1$, $c'_2=3/2=c_2$ and $B'=B$, the paths of
integration $L(c,B)$ will be moved across the poles
$s$, $1$ and $\rho$ to the line $L(0)$, the imaginary axis.
We first describe what happens when we move the
$w$-path: 
\begin{equation} \label{eq:6.2}
T^\la(s) = \int_{L(c,B)}\cdots\,dz\int_{L(0)}\cdots\,dw
+U^\la_2(s)=U^\la_1(s) +U^\la_2(s),
\end{equation}
say, where by the residue theorem
\begin{equation} \label{eq:6.3}
U^\la_2(s) = \int_{L(c,B)}
\Ga(z-s)\frac{\ze'(z)}{\ze(z)}\,J(z,s)dz,
\end{equation}
with
\begin{align} \label{eq:6.4}
J(z,s) &=
\frac{\ze'(s)}{\ze(s)}\,M^\la(z-s)\cos\{\pi(z-s)/2\} 
\notag \\ &
\quad -\Ga(1-s)M^\la(z+1-2s)\cos\{\pi(z-1)/2\} \\ &
\quad +\sum_\rho\,\Ga(\rho-s)M^\la(z+\rho-2s)
\cos\{\pi(z-\rho)/2\}.\notag
\end{align}
By the usual estimates and Lemma \ref{lem:4.2} the
repeated integral $U^\la_1(s)$ defines a holomorphic
function for $1/2<\si<3/2$, hence by Theorem
\ref{the:5.1}, the function $U^\la_2(s)$ must have an
analytic continuation to the same strip! 

For $1<\si<3/2$ the function $J(z,s)$ is holomorphic in
$z$ on and between the paths $L(c,B)$ and $L(0)$. If we
define $J(z,s)$ for $z\in L(c,B)$ by continuity at 
$s=1$ and points $s=\rho$, it becomes holomorphic in $s$
for $3/4<\si<3/2$; apparent poles cancel each other. 
What can we say about the integral for $U^\la_2(s)$? The
critical part is the one that 
corresponds to the sum over $\rho=\be\pm i\ga$ in
(\ref{eq:6.4}). For $|y|>B\ge 2$ its integrand is majorized
by
$$\sum_\rho |y|^{c_1-\si-1/2}(\log|y|)|\rho|^{\be-\si-1/2}
(|y+\rho|+1)^{-c_1-\be+2\si-7/2}.$$
Since $c_1=1$, $0<\be<1$ and $|\rho_n|\sim 2\pi n/\log n$ as
$n\to\infty$, the analog of Lemma \ref{lem:4.2} for the
integral of a sum proves the absolute convergence and
holomorphy of the integral when $3/4<\si<3/2$.

To justify the above application of the residue theorem
one may start with $w$-integrals over a sequence of
closed contours $W_R$, $B<R=R_k\to\infty$, whose upper part
is shown in Figure \ref{fig:2}. 
\begin{figure}[htb]
$$\includegraphics[width=4cm]{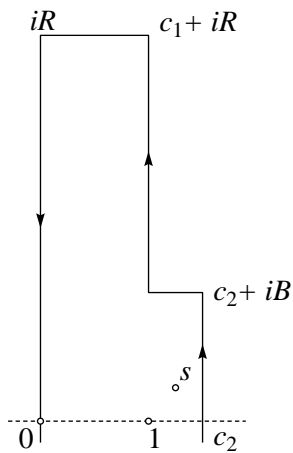}
$$
\caption{Upper half of $W_R$ \label{fig:2}}
\end{figure}
Here the numbers $R=R_k$ are chosen `away from the numbers
$\ga_n$', in the sense that on the horizontal segments
$\{v=\pm R\}$ one has $\ze'(w)/\ze(w)\ll \log^2|v|$.
One may require that $R_k\in(k,k+1)$, $k=1,2,\cdots$; 
cf.\ the expansion of $(\ze'/\ze)(\cdot)$ and ways of
estimating this quotient in Titchmarsh \cite{Ti86}. One can
now use the standard majorants to show that for
$3/4<\si<3/2$,
\begin{align} & \label{eq:6.5}
\int_{L(c,B)}\Ga(z-s)\frac{\ze'(z)}{\ze(z)}\,dz
\int_{c_1+iR}^{iR}\Ga(w-s)\frac{\ze'(w)}{\ze(w)}\,
\cdot \notag \\ & \quad\cdot
\,M^\la(z+w-2s)\cos\{\pi(z-w)/2\}dw\to
0\quad\mbox{as}\;\;R\to \infty.
\end{align}
Similarly for the segment where $v=-R$. We summarize what
we have found so far:
\begin{proposition} \label{prop:6.1}
For $1<\si<3/2$ and $|\tau|<B$ the integral for $T^\la(s)$
admits a holomorphic decomposition $T^\la(s)
= U^\la_1(s)+U^\la_2(s)$, see $(\ref{eq:6.2})$, in which
both integrals $U^\la_j(s)$ define analytic functions for
$3/4<\si<3/2$.
\end{proposition}

\medskip
In a second step, inverting order of integration where
necessary, the $z$-paths $L(c,B)$ in the
integrals (\ref{eq:6.2}), (\ref{eq:6.3}) for $U^\la_1(s)$
and $U^\la_2(s)$ are moved to $L(0)$. Again using the
residue theorem, this results in the decomposition
\begin{align} \label{eq:6.6} 
T^\la(s) &=
\int_{L(0)}\,\cdots\,dz\int_{L(0)}\,\cdots\,dw
+2\int_{L(0)}\,\Ga(z-s)\frac{\ze'(z)}{\ze(z)}\,
J(z,s)dz \notag \\ & 
+\bigg\{\frac{\ze'(s)}{\ze(s)}\,J(s,s) - \Ga(1-s)J(1,s)+
\sum_{\rho'}\,\Ga(\rho'-s)J(\rho',s)\bigg\}.
\end{align}
The new integrals will define holomorphic functions for
$0<\si<1$. In the case of the single integral involving
the sum over $\rho=\be+i\ga$ in $J(z,s)$, where $\be$ might
be close to $1$, this can be seen by moving the integral
over to $L(-\eta)$ with variable $\eta\in(0,1)$. On the 
interval $\{1/2\le s\le 3/4\}$ the integrals are $\cal{O}(1)$ 
as $\la\to\infty$.

We now turn to the {\it big residue} in the last line of
(\ref{eq:6.6}). With the aid of (\ref{eq:6.4}) it leads to
nine terms. Four of these combine into the three terms of
$\Si^\la(s)$ in (\ref{eq:3.1}); here the sum
$\Si^\la_2(s)$ of the double series has to be formed in a
suitable way as described below. Another term gives the
important constituent 
\begin{equation} \label{eq:6.7}
U^\la_3(s)\stackrel{\mathrm{def}}{=}\Ga^2(1-s)M^\la(2-2s).
\end{equation}
By Proposition \ref{prop:2.1} the
function $U^\la_3(s)$ is meromorphic for $0<\si<1$, with
just one pole, a first-order pole at $s=1/2$. Since
$M^\la(z)$ has residue $-(2\la/\pi)A^E$ at $z=1$,
expansion about $s=1/2$ gives
\begin{equation} \label{eq:6.8}
U^\la_3(s)=\frac{A^E\la}{s-1/2}+H^\la_2(s),
\end{equation}
where $A^E=\int_0^1 E(\nu)d\nu$ and $H^\la_2(s)$ is
holomorphic for $0<\si<1$. On the interval $\{1/2\le s\le
3/4\}$ one has $H^\la_2(s)=\cal{O}(\la\log\la)$ as $\la\to
\infty$. 

The other four terms provided by the last line of
(\ref{eq:6.6}) combine to
\begin{align} \label{eq:6.9}
U^\la_4(s) &=
-2\Ga(1-s)\frac{\ze'(s)}{\ze(s)}\,M^\la(1-s)\sin(\pi s/2)
\notag \\ & \quad - 
2\Ga(1-s)\sum_\rho\,\Ga(\rho-s)M^\la(1+\rho-2s)
\sin(\pi\rho/2).
\end{align}
This function is meromorphic for $0<\si<1$ with poles
at zeta's complex zeros that cancel one another, and
further poles at the points $s=\rho/2$. Thus $U^\la_4(s)$
is holomorphic for $1/2\le\si<1$, and for $1/4<\si<1$ under
RH. For $s\searrow 1/2$ one has $U^\la_4(s)=\cal{O}(\la)$
as $\la\to \infty$.

To justify the application of the residue theorem one now
has to show that (\ref{eq:6.5}) remains valid if $L(c,B)$
is replaced by $L(0)$, and also that 
$$\int_{1+iR}^{iR}\Ga(z-s)\frac{\ze'(z)}{\ze(z)}\,
J(z,s)dz\to 0\quad\mbox{as}\;\;R=R_k\to
\infty.$$ 
This is no problem if $1/2<\si<3/2$. Similarly for the
segment $v=-R$. 

Since $T^\la(s)$ is holomorphic for $\si>1/2$ by Theorem
\ref{the:5.1}, the discussion above implies that the
`big residue' in (\ref{eq:6.6}) represents a holomorphic
function for $1/2<\si<1$, provided the sum
over $\rho'$ is interpreted as 
$$\lim_{R=R_k\to\infty}\,\sum_{|{\rm Im\,}\rho'|<R}\,
\Ga(\rho'-s)J(\rho',s).$$
It follows that the function $\Si^\la(s)$ in (\ref{eq:3.1})
likewise represents a holomorphic function for
$1/2<\si<1$, provided the double sum $\Si^\la_2(s)$
involving $\rho$ and $\rho'$ is interpreted as a limit of
$\sum_{|{\rm Im\,}\rho'|<R}\,\sum_{\rho}$.
Under RH the double series is absolutely convergent; cf.\
Lemma \ref{lem:4.2}. Summarizing we have 
\begin{theorem} \label{the:6.2}
For $s=\si+i\tau$ with $1/2<\si<1$ there is a
holomorphic decomposition 
\begin{equation} \label{eq:6.10}
T^\la(s) = \frac{A^E\la}{s-1/2} + \Si^\la(s) +
H^\la_3(s),\quad A^E=\int_0^1 E(\nu)d\nu,
\end{equation}
where $\Si^\la(s)$ is given by $(\ref{eq:3.1})$ with
proper interpretation of the double sum. The function
$H^\la_3(s)$ has an analytic continuation to the strip
$\{1/2\le\si<1\}$, and to the strip $\{1/4<\si<1\}$ under
RH. On the interval $\{1/2\le s\le 3/4\}$ one has
$H^\la_3(s)=\cal{O}(\la\log\la)$ as $\la\to \infty$.
\end{theorem}

\setcounter{equation}{0}    
\section{Proofs for the main results in Section \ref{sec:3}}
\label{sec:7} 
\begin{proof}[Proof of Theorem \ref{the:3.1}] The
proof is obtained by combining Theorem \ref{the:5.2},
Lem\-ma \ref{lem:5.3} and Theorem \ref{the:6.2}. For
$1/2<\si<1$ these results give the following holomorphic
representations for the sum $V^\la(s)$:
\begin{align} \label{eq:7.1}
V^\la(s) &= 2\sum_{0<
2r\le\la}\,E(2r/\la)D_{2r}(s) 
= T^\la(s)-D_0(s)+H^\la_4(s) \notag \\ & 
= T^\la(s)-\frac{1/4}{(s-1/2)^2}+H^\la_5(s) \\ &=
-\frac{1/4}{(s-1/2)^2}+\frac{A^E\la}{s-1/2}+\Si^\la(s)
+H^\la_6(s).\notag
\end{align}
Here $A^E=\int_0^1 E(\nu)d\nu$ and $\Si^\la(s)$ is given by
(\ref{eq:3.1}) with suitable interpretation of the double
sum. The functions $H^\la_j(s)$ are holomorphic for 
$1/2\le\si<1$, and for $1/4<\si<1$ under RH. For the final
line of (\ref{eq:3.2}) one applies (\ref{eq:7.1}) to
$V^1(s)=0$ and subtracts the result from (\ref{eq:7.1}) for
$V^\la(s)$. 

Results on $H^\la_1(s)$ in Theorem \ref{the:5.2} and
$H^\la_3(s)$ in Theorem \ref{the:6.2} show that on the
interval $\{1/2\le s\le 3/4\}$ one has
$H^\la_6(s)=\cal{O}(\la\log\la)$ as $\la\to\infty$.
We will see below that the sum of the double series
$\Si^\la_2(s)$ in the definition of $\Si^\la(s)$ may be
formed as a limit of square partial sums.
\end{proof}
\noindent{\scshape Discussion of $\Si^\la(s)$.} 
The part of the double sum $\Si^\la_2(s)$ in which
$\ga={\rm Im}\,\rho$ and $\ga'={\rm Im}\,\rho'$ have the
same sign defines a meromorphic function for $1/2<\si<1$
whose only poles occur at complex zeros of $\ze(\cdot)$.
Indeed, for $s$ different from those zeros the series is
absolutely convergent; to verify this one may use an analog
for sums of the final part of Lemma \ref{lem:4.2}. 
\begin{proposition} \label{prop:7.1}
The function $\Si^\la_2(s)$ can be obtained as a limit of
square partial sums:
\begin{align} & \Si^\la_2(s) =\notag \\ &
\lim_{R\to\infty}\,\sum_{|\ga|,\,|\ga'|<R}
\,\Ga(\rho-s)\Ga(\rho'-s)
M^\la(\rho+\rho'-2s)\cos\{\pi(\rho-\rho')/2\},\notag
\end{align}
where $R$ `stays away' from the numbers $\ga_n$ as
described in Section \ref{sec:6}.
\end{proposition}
\begin{proof} By the preceding we may restrict ourselves to
the double sum $\Si^\la_3(s)$ in which $\ga$ and $\ga'$
have opposite sign; as shown in Section \ref{sec:6} it is
the limit of $\sum_{|\ga'|<R}\,\sum_{\ga;\,\ga\ga'<0}$ for
suitable $R\to\infty$. In order to prove that one can use
square partial sums it will suffice to show that for
fixed $s$ (different from the numbers $\rho'$)
with $1/2<\si<1$,
$$\sum_{-R<\ga'<0}\,\sum_{\ga>R}\,\Ga(\rho'-s)\Ga(\rho-s)
M^\la(\rho+\rho'-2s)\cos\{\pi(\rho-\rho')/2\}\to 0$$
as $R\to\infty$. Using standard
majorants and setting $\ga'=-\ga''$ this will follow if we
prove
\begin{equation} \label{eq:7.2}
\sum_{0<\ga''<R,\,\ga>R}\,(\ga\ga'')^{(1/2)-\si}
(\ga-\ga''+1)^{2\si-7/2}\to 0
\end{equation}
as $R\to\infty$. Now the number of zeta's zeros with
$T<\ga\le T+1$ is $\cal{O}(\log T)$, cf.\ Titchmarsh
\cite{Ti86}. Thus for fixed $\si\in(1/2,1)$ the double sum
in (\ref{eq:7.2}) is majorized by 
\begin{align} &
\int_2^R v^{(1/2)-\si}(\log v)dv\int_R^\infty
y^{(1/2)-\si}(y-v+1)^{2\si-7/2}(\log y)dy \notag \\ &
\ll \int_2^R v^{(1/2)-\si}(\log v)R^{(1/2)-\si}(\log R)
(R-v+1)^{2\si-5/2}dv \notag \\ & \ll
R^{(1/2)-\si}\log^2 R\int_2^R
v^{(1/2)-\si}(R-v+1)^{2\si-5/2}dv.\notag
\end{align}
In the last integral one may treat the $v$-intervals
$(1,R/2)$ and $(R/2,R)$ separately to obtain the final
majorant
$$(R^{-1/2}+R^{1-2\si})\log^2 R,\quad\mbox{which}\;\;\to
0\quad\mbox{as}\;\;R\to\infty.$$
\end{proof}
\begin{proof}[Proof of Corollary \ref{cor:3.2}] 
The proof uses induction with respect to $m\ge 1$. The
induction hypothesis is that the PPC for pairs
$(p,\,p+2r)$ is known to hold for every $r<m$. 

(i) Suppose that the PPC is also true for $r=m$. Then if
$\la$ is any given number in $(2m,2m+2]$, the PPC is true
for every $r<\la/2$. Hence for any smooth $E$, by
(\ref{eq:2.4}) the function 
\begin{equation} \label{eq:7.3}
W^\la(s) \stackrel{\mathrm{def}}{=}
2\sum_{0<2r<\la}\,E(2r/\la)
\Big(D_{2r}(s)-\frac{C_{2r}}{s-1/2}\Big)
\end{equation}
has good (local pseudofunction) boundary behavior as
$\si\searrow 1/2$. Now by (\ref{eq:3.2})--(\ref{eq:3.4}) 
one has, for the present $\la$,
$$G^\la(s)=W^\la(s)-H^\la(s),$$
where $H^\la(s)$ is holomorphic for $1/2\le\si<1$. Hence
$G^\la(s)$ will also have good boundary behavior. 

(ii) Conversely suppose that $G^\la(s)$ has good
boundary behavior as $\si\searrow 1/2$ for some smooth $E$
and some $\la\in(2m,2m+2]$. Then with this $\la$, the sum
$W^\la(s)$ in (\ref{eq:7.3}) has good boundary behavior. But
we know from the induction hypothesis that 
$$2\sum_{0<2r<2m}\,E(2r/\la)
\Big(D_{2r}(s)-\frac{C_{2r}}{s-1/2}\Big)$$
has good boundary behavior, hence so does the difference
$$2E(2m/\la)\Big(D_{2m}(s)-\frac{C_{2m}}{s-1/2}\Big).$$
Since $E(2m/\la)\ne 0$ this implies the PPC for pairs
$(p,\,p+2m)$.
\end{proof}

\setcounter{equation}{0}   
\section{Proof of Theorem \ref{the:3.4}} 
\label{sec:8}
Assume RH and let $N(T)$ denote the number of zeta's zeros
$\rho=(1/2)+i\ga$ with $0<\ga\le T$. Setting $s=(1/2)+\de$
with $0<\de<1/4$ we will write $\Si^\la_0(\de)
\stackrel{\mathrm{def}}{=}\Si^\la_3(s)$ as an integral.
Recall that $\ga={\rm Im}\,\rho$ and $\ga'={\rm Im}\,\rho'$
in $\Si^\la_3(s)$ have opposite sign. Thus it is convenient
to introduce
\begin{align} & \label{eq:8.1}
F^\la(y,v,\de) = \\ &
\Ga(iy-\de)\Ga(-iv-\de)M^\la\{-2\de+i(y-v)\}
\cosh\{\pi(y+v)/2\}.\notag
\end{align}
Since $F^\la(y,v,\de)=F^\la(-v,-y,\de)$ one finds that
\begin{equation} \label{eq:8.2}
\Si^\la_0(\de)=\Si^\la_3(s)=
2\int\int_{y,\,v>2} F^\la(y,v,\de)dN(y)dN(v).
\end{equation}
For the study of $F^\la(\cdot)$ we use Stirling's uniform
asymptotic formula for $|\arg z|<\pi-\eps$ and $|z|>2$:
\begin{equation} \label{eq:8.3}
\log\Ga(z)=(z-1/2)\log z-z+(1/2)\log(2\pi)+\cal{O}(1/|z|);
\end{equation}
cf.\ Whittaker and Watson \cite{WW27}. It shows that for
$|y|>2$
\begin{align} 
\Ga(iy-\de) & =\sqrt{2\pi}e^{-\pi i({\rm
sgn}\,y)(1+2\de)/4}|y|^{-\de-1/2}
\big[1+\cal{O}(1/|y|)\big]\cdot\notag \\ &
\quad\cdot\exp\{iy\log |y|-iy-\pi|y|/2\}.\notag
\end{align}
This formula will be used also for $\Ga(-iv-\de)$. Thus by
(\ref{eq:8.1}), (\ref{eq:2.8}) and (\ref{eq:2.9}), for
$y,\,v>2$, 
\begin{align} \label{eq:8.4}
& F^\la(y,v,\de) =\pi(yv)^{-\de-1/2}
e^{i(y\log y-y-v\log v +v)}\,\cdot \notag\\
& \;\;\cdot\la^{-2\de+i(y-v)}M\{-2\de+i(y-v)\}
\big[1+\cal{O}(1/y)+\cal{O}(1/v)\big] \\ & 
\ll G^\la(y,v,\de)\stackrel{\mathrm{def}}{=} \la^{-2\de}
(yv)^{-\de-1/2}(|y-v|+1)^{2\de-7/2}.\notag
\end{align}

To complete the proof of Theorem \ref{the:3.4} we establish
\begin{proposition} \label{prop:8.1}
The pole-type behavior of $\Si^\la_0(\de)=\Si^\la_3(s)$ as
$\de\searrow 0$ is the same as that of the reduced function
\begin{align} \label{eq:8.5}
\Si^\la_4(s) = \Si^\la_*(\de) &= 2\pi\int\int_{y,\,v>2;\,
|y-v|<y^{1/2}}\,
y^{-1-2\de+i(y-v)}\,\cdot \notag \\ &\qquad\qquad\cdot
M^\la\{-2\de+i(y-v)\}dN(y)dN(v).
\end{align}
\end{proposition}
\begin{proof}
In the discussion of the integral $I^\la$ of $F^\la$ one may
ignore the quantities $\cal{O}(1/y)$ and $\cal{O}(1/v)$ that
occur in (\ref{eq:8.4}); by Lemma \ref{lem:4.2} they lead
to bounded functions of $\de$. Furthermore, it follows
from the majorization (\ref{eq:8.4}) that the integral
$I^\la_1$ of $|F^\la(y,v,\de)|dN(y)dN(v)$ over the set
$\Om_1$ where $y,\,v>2$ and $|y-v|\ge y^{1/4}$ is bounded on
the interval $\{0<\de<1/4\}$. Indeed, for fixed $\la$, 
\begin{align} 
I^\la_1 &\ll \int\int_{\Om_1}\,G^\la(y,v,\de)
dN(y)dN(v)\notag\\ & \ll \int_2^\infty
y^{-\de-1/2}(\log y)dy\int_{|v-y|\ge
y^{1/4}}\,(|v-y|+1)^{2\de-7/2}(\log v)dv\notag \\ &\ll
\int_2^\infty y^{-\de-1/2}(\log y)\cdot y^{(\de/2)-5/8}(\log
y)dy. \notag
\end{align}

It follows that we may surely restrict ourselves to the part
$I^\la_2$ of the integral in (\ref{eq:8.2}) over the set
$\Om_2$ where $y,\,v>2$ and $|y-v|<y^{1/2}$. On this set the
function
$$v^{-\de-1/2}=y^{-\de-1/2}\{1+(v-y)/y\}^{-\de-1/2}=
y^{-\de-1/2}+\cal{O}(y^{-\de-1})$$
may be replaced by $y^{-\de-1/2}$; by Lemma \ref{lem:4.2}
the error term gives rise to a bounded function of $\de$. We
finally observe that on $\Om_2$
$$y\log y-y-v\log v +v=\int_v^y(\log t)dt=(y-v)(\log y)
+\cal{O}\{|y-v|^2/y\},$$
hence 
$$e^{i(y\log y-y-v\log v +v)}=y^{i(y-v)}\big[1+
\cal{O}\{|y-v|^2/y\}\big].$$
The contribution to $I^\la_2$ due to the final
$\cal{O}$-term is bounded on the interval $\{0<\de<1/4\}$.
Thus as regards pole-type behavior, the function
$\Si^\la_0(\de)$ can be reduced to $\Si^\la_*(\de)$ or
$\Si^\la_4(s)$.
\end{proof}
\begin{remark} \label{rem:8.2}
The pole-type behavior of $\Si^\la_3(s)$ and $\Si^\la(s)$
as $s\searrow 1/2$ is also the same as that of the
symmetric sum
$$2\pi\sum_{\ga,\,\ga';\,|\ga-\ga'|<(\ga\ga')^{1/6}}
(\ga\ga')^{-s+i(\ga-\ga')/2}M^\la\{1-2s+i(\ga-\ga')\}.$$
\end{remark}

\setcounter{equation}{0}   
\section{Pair correlation of zeta's zeros and\\
the conditional Theorem \ref{the:3.6}} 
\label{sec:9}
Goldston, Pintz and Yildirim \cite{GPY05} have
shown conditionally that there are infinitely many prime
pairs $(p,\,p+2r)$ for some $r$ with $2r\le 16$. Their
proof used a hypothesis of Elliott and
Halberstam \cite{EH68} on (weighted) equidistribution of
primes in arithmetic progressions. Below we will discuss the
conditional Theorem \ref{the:3.6} which would imply that
there is an abundance of prime pairs for some difference
$2r$. The proof depends on Hypothesis
\ref{hyp:3.5}, which is supported by Theorem \ref{the:3.4}
and related aspects of the pair correlation of zeta's complex
zeros.

As an introduction we describe some results on the pair
correlation from the now extensive literature. In the
background was a comparison, under RH, of the fine
distribution of zeta's zeros $\rho=(1/2)+i\ga$ to the
eigenvalue distribution of large unitary matrices; see the
LMS Lecture Notes vol.\ 322 \cite{LMS05}. Using an ingenious
computation, Montgomery \cite{Mo73} obtained the following
basic pair-correlation result, cf.\ Goldston and Montgomery
\cite{GM87}:
\begin{theorem} \label{the:9.1} 
Assume RH and set $w(u)=4/(4+u^2)$ so that $w(0)=1$. Then for
$T\to\infty$, uniformly for $\al\in[0,1]$, 
\begin{align} \label{eq:9.1}
F_w(\al,T) & \stackrel{\mathrm{def}}{=}
\frac{2\pi}{T\log T}\sum_{0<\ga,\,\ga'\le
T}\,e^{i\al(\ga-\ga')\log T}w(\ga-\ga')\notag\\
&=\{1+o(1)\}T^{-2\al}\log T+\al+o(1).
\end{align}
\end{theorem}
One may speculate that (\ref{eq:9.1}) holds for many other
weight functions $w$ with $w(0)=1$; cf.\  Hejhal \cite{He94},
Rudnick and Sarnak \cite{RS94}, \cite{RS96}, 
Bogomolny and Keating \cite{BK95}. Furthermore, Montgomery
used the PPC to support the conjecture that, uniformly
for $1\le\al\le C$, 
\begin{equation} \label{eq:9.2}
F_w(\al,T)=1+o(1)\quad\mbox{as}\;\;T\to\infty.
\end{equation}
This conjecture would imply that almost all of zeta's
complex zeros are simple: $N_s(T)\sim N(T)\sim(T\log
T)/(2\pi)$. It also implies that the behavior of
$F_w(\al,T)$ for $\al\ge 1$ is determined largely by the
terms in the double sum of (\ref{eq:9.1}) for which
$\ga'=\ga$; the terms with $\ga'\ne\ga$ would essentially
cancel each other. To visualize the exponentials
$e^{i\al(\ga-\ga')\log T}$ on the unit circle, observe that
the mean spacing of zeta's zeros $(1/2)+i\ga$ for
$\ga$ near $T$ is approximately $2\pi/\log T$. 

See also the subsequent work by Gallagher and Mueller
\cite{GaM78}, Heath-Brown \cite{HB82}, Gallagher
\cite{Ga85}, Goldston \cite{Go88}, Goldston and Gonek
\cite{GG90}, \cite{GG98}, Goldston, Gonek,
\"{O}zl\"{u}k and Snyder \cite{GGOS00}, Montgomery and 
Soundararajan \cite{MS04}, Chan \cite{Cha03}, \cite{Cha04a},
\cite{Cha04b}, and Goldston \cite{Go05}.

\smallskip
Always assuming RH, it is interesting to compare the case
$\al=1$ of Montgomery's result and the case $\la=1$ of
Theorem \ref{the:3.4}. By (\ref{eq:9.1})
$$2\pi\sum_{0<\ga,\,\ga'\le
T}\,e^{i(\ga-\ga')\log T}w(\ga-\ga')\sim T\log
T\quad\mbox{as}\;\;T\to\infty,$$
while by (\ref{eq:7.1}) (since $V^1(s)=0$) and
Theorem \ref{the:3.4} (with $2s=1+\de$) 
$$2\pi\sum_{\ga,\,\ga';\,|\ga-\ga'|<\ga^{1/2}}\,
\ga^{-1-\de+i(\ga-\ga')}M^\la\{-\de
+i(\ga-\ga')\}\sim 1/\de^2$$
as $\de\searrow 0$. It appears that in first approximation,
the behavior of the second sum is also determined by the
terms with $\ga'=\ga$:
$$2\pi\int_2^\infty
y^{-1-\de}M^\la(-\de)dN(y)\sim\int_2^\infty y^{-1-\de}(\log
y)dy\sim 1/\de^2.$$
 
\begin{proof}[Support for Hypothesis \ref{hyp:3.5}]
Using Theorems \ref{the:3.1} and \ref{the:3.4} with
$s=(1/2)+\de$, and writing $\Si^\la_4(s)=\Si^\la_*(\de)$ as
in (\ref{eq:8.5}), we will now consider the difference
\begin{align} \label{eq:9.3}
\Si^\la_*(\de)-\Si^1_*(\de) &=  2\pi
\sum_{\ga,\,\ga';\,|\ga-\ga'|<\ga^{1/2}}\,
\big\{\la^{-2\de+i(\ga-\ga')}-1\big\}\cdot \notag \\ &
\qquad\qquad\cdot\ga^{-1-2\de+i(\ga-\ga')}M\{-2\de
+i(\ga-\ga')\}.
\end{align}
For $\la=2$ it follows from Theorem
\ref{the:3.1} (since $V^2(s)=0$) that
$$\Si^2_*(\de)-\Si^1_*(\de)=-\frac{A^E}{\de}+\cal{O}(1)
\quad\mbox{as}\;\;\de\searrow 0.$$
Compared to the original sum $\Si^1_*(\de)$, the general term
in (\ref{eq:9.3}) now contains an additional factor
$2^{-2\de+i(\ga-\ga')}-1$. For small $\de$ and $\ga-\ga'$,
this factor is like
$\{-2\de+i(\ga-\ga')\}\log 2$. If one may ignore the
contribution due to larger $|\ga-\ga'|$, the effect of
the factor will be roughly 
$$\{-2\de(\log 2)/(4\de^2)+\,\mbox{contribution of}\;\;
i(\ga-\ga')(\log 2)/(4\de^2).$$
We know that the new pole is $-A^E/\de$, hence the second
contribution must also result in a first order pole with
modest residue.

In the case of general $\la$ the effect of the factor
$$\la^{-2\de+i(\ga-\ga')}-1\approx \{-2\de+i(\ga-\ga')\}\log
\la$$
might well be a first order pole with residue of order
$\log\la$; cf.\ also (\ref{eq:2.6}) and (\ref{eq:1.6}). Thus
Hypothesis
\ref{hyp:3.5} appears to be  plausible.
\end{proof}
\begin{proof}[Proof of Theorem \ref{the:3.6}] For given
$\eps\in(0,1)$ we form a smooth sieving function
$E^\la(\nu)=E(\nu/\la)$ (as in Section \ref{sec:2}) such
that
\begin{equation} \label{eq:9.4}
A^E=\int_0^1 E(\nu)d\nu > 1-\eps/3.
\end{equation}
For $\la>2$ and $s\in(1/2,1)$ we now use the
final representation for $V^\la(s)$ in Theorem \ref{the:3.1}:
\begin{align} \label{eq:9.5}
V^\la(s) &= 2\sum_{0<2r\le\la}\,E(2r/\la)D_{2r}(s)
\notag \\ &= \frac{A^E(\la-1)}{s-1/2}+\Si^\la(s)-\Si^1(s)+
H^\la(s),
\end{align}
where $H^\la(s)$ is holomorphic for $1/2\le\si<1$. Thus by
Hypothesis \ref{hyp:3.5}
\begin{equation} \label{eq:9.6}
\limsup_{s\searrow
1/2}\,(s-1/2)V^\la(s)=A^E(\la-1)+\om(\la),
\end{equation}
where $\om(\la=o(\la)$ as $\la\to\infty$. Hence by
(\ref{eq:9.4}), the right-hand side of (\ref{eq:9.6}) will
be greater than $(1-\eps/2)\la$ for all sufficiently large
$\la$. We choose the smallest even positive integer
$\la=2m$ for which this is so. Since $0\le E(\nu)\le 1$, 
it then follows from (\ref{eq:9.5}) that
\begin{equation} \label{eq:9.7}
\limsup_{s\searrow
1/2}\,(s-1/2)\sum_{r\le m}\,D_{2r}(s)>(1-\eps/2)m.
\end{equation}
We will show that this implies
\begin{equation} \label{eq:9.8}
\limsup_{x\to\infty}\,(1/x)\sum_{r\le m}\,\psi_{2r}(x)>
C=(2-\eps)m.
\end{equation}

Suppose to the contrary that
$(1/x)\sum_{r\le m}\,\psi_{2r}(x)\le C$ for all $x\ge x_0\ge
1$. Now by (\ref{eq:2.3}) and (\ref{eq:2.1}) for
$s\in(1/2,1)$,
\begin{align} 
D_{2r}(s) &= \int_1^\infty x^{-s}(x+2r)^{-s}d\psi_{2r}(x)
\le \int_{x_0}^\infty x^{-2s}d\psi_{2r}(x)+\cal{O}(1)
\notag \\ &\le 2s\int_{x_0}^\infty
x^{-2s-1}\psi_{2r}(x)dx+\cal{O}(1).\notag
\end{align}
Hence it would follow that
\begin{align} & \label{eq:9.9}
\sum_{r\le m}\,D_{2r}(s)\le 2s\int_{x_0}^\infty
x^{-2s-1}\sum_{r\le m}\,\psi_{2r}(x)dx+\cal{O}(1) \notag\\
&\le 2s\int_{x_0}^\infty Cx^{-2s}dx+\cal{O}(1)
\le 2sC/(2s-1)+\cal{O}(1).
\end{align}
As a result the upper residue in (\ref{eq:9.7}) would
be $\le C/2$, hence $\le(1-\eps/2)m$. This contradiction
proves (\ref{eq:9.8}).

In order to pass from (\ref{eq:9.8}) to (\ref{eq:3.7}) one
may appeal to the discussion in Section \ref{sec:2}.
\end{proof}

\setcounter{equation}{0} 
\section{Positivity of sums $\Si^\la_2(s)$ and \\
conditional abundance of prime pairs} \label{sec:10}
We will verify the positivity of the double sums
$\Si^\la_2(s)$ in (\ref{eq:3.1}) for $1/2<s<1$ when $\hat
E^\la(t)\ge 0$.  
\begin{proof}[Proof of Proposition \ref{prop:3.7}]
It will be convenient to replace $\rho'$ in the double sum 
$\Si^\la_2(s)$ by $\overline\rho'$. Set
$\Om_R(t,s)=\Om'_R(t,s)+\Om''_R(t,s)$, where
$$ \Om'_R(t,s)= \sum_{|{\rm Im}\,\rho|,\,|{\rm
Im}\,\rho'|<R}\,\Ga(\rho-s)
\Ga(\overline\rho'-s)t^{2s-\rho-\overline\rho'}
\cos(\pi\rho/2)\cos(\pi\overline\rho'/2),$$
and $\Om''_R(t,s)$ is the corresponding function with $\sin$
instead of $\cos$. Then
$$\Om'_R(t,s)=\bigg|\sum_{|{\rm Im}\,\rho|<R}
\Ga(\rho-s)t^{s-\rho}\cos(\pi\rho/2)\bigg|^2\ge 0,$$
and similarly for $\Om''_R(t,s)$. Hence by Proposition
\ref{prop:7.1} and (\ref{eq:2.7})
$$\Si^\la_2(s)=\frac{1}{\pi}\lim_{R\to\infty}\,\int_0^\infty
\hat E^\la(t)\Om_R(t,s)dt\ge 0.$$
\end{proof}

One may use Proposition \ref{prop:3.7} to derive another
conditional abundance result:
\begin{theorem} \label{the:10.2}
Suppose that for certain positive integers
$m_1<m_2<\cdots<m_k$ there are a constant $c>0$ and a
sequence $S$ of numbers $\la\to \infty$, such that for
$\la\in S$ and sufficiently large $x$,
say $x\ge x_1=x_1(\la)$ with $\log x_1(\la)=o(\la)$, one has
\begin{equation} \label{eq:10.1}
\frac{1}{k}\,\sum_{j=1}^k\,\pi_{2m_j}(x)\ge
c\cdot\frac{2}{\la}\sum_{0<2r\le\la}\,\pi_{2r}(x).
\end{equation}
Then 
\begin{equation} \label{eq:10.2}
\limsup_{x\to\infty}\,\frac{1}{k}\,\sum_{j=1}^k\,
\frac{\pi_{2m_j}(x)}{x/\log^2 x}\ge c.
\end{equation}
\end{theorem}
There is both heuristic and numerical support for the
hypothesis of the theorem. The proof below makes use of the
sieving function $E^\la(\nu)=E^\la_F(\nu)$. Although it does
not satisfy the smoothness requirement imposed in Section
\ref{sec:2}, one can show that it may be used anyway; it
gives a better result here than $E^\la_J$. We plan to
return to the details later; cf.\ also \cite{Ko07}.
\begin{proof}[Brief indication of the proof] It suffices to
treat the case $k=1$, the general case being similar; we
write $m_k=m$. Now Theorem \ref{the:3.1} with $E=E_F$, so that
$A^E=1/2$, the decomposition $\Si^\la(s)=\Si^\la_1(s)+\Si^\la_2(s)$
and Proposition \ref{prop:3.7} imply the following inequality
for $1/2<s<3/4$:
$$ 2\sum_{0<2r\le\la}\,E(2r/\la)
D_{2r}(s) \ge -\frac{1/4}{(s-1/2)^2}+
\frac{\la/2}{s-1/2}+\Si^\la_1(s)-\cal{O}(\la\log\la).$$
Here $0\le E(2r/\la)\le 1$ and the sum $\Si^\la_1(s)$ of the
first two terms in (\ref{eq:3.1}) is $\cal{O}(\la^{1/2})$.
Setting $s-1/2=\de$ it follows that
\begin{equation} \label{eq:10.3}
\frac{2}{\la}\sum_{0<2r\le\la}\,\de
D_{2r}\{(1/2)+\de\}\ge
\frac{1}{2}-\frac{1}{4\la\de}-\cal{O}(\de\log\la).
\end{equation}
Combining this with the hypothesis of the theorem,
appropriate estimates show that 
\begin{equation} \label{eq:10.4} 
\de D_{2m}\{(1/2)+\de\}
\ge c\bigg(\frac{1}{2}-\frac{1}{4\la\de}\bigg)
-\cal{O}\{\de\log(x_1(\la)\}.
\end{equation}
For given $\eps\in(0,1/2)$ we now choose $\la\to\infty$ in $S$ and 
$\de\searrow 0$ in $(0,1/4)$ such that $1/(4\la\de)=\eps/2$. Since
$\log x_1(\la)=o(\la)$ one may conclude that
\begin{equation} \label{eq:10.5}
\limsup_{\de\searrow 0}\,\de
D_{2m}\{(1/2)+\de\}\ge(1-\eps)c/2.
\end{equation}
From here on one may argue as in the proof of Theorem
\ref{the:3.6}.
\end{proof}

\bigskip

\noindent{\scshape KdV Institute of Mathematics, University of 
Amsterdam, \\
Plantage Muidergracht 24, 1018 TV Amsterdam, Netherlands}

\noindent{\it E-mail}: {\tt korevaar@science.uva.nl}


\begin{thebibliography}{[41]} 

\bibitem{Ar04} R.\ F.\ Arenstorf, {\it `There are infinitely
many  prime twins'}, available on the Internet at
http://arxiv.org/abs/math/0405509v1. Article posted May 26,
2004; withdrawn June 9, 2004. [{\it sec \ref{sec:1},
\ref{sec:2}}]

\bibitem{BK95} E.\ B.\ Bogomolny and J.\ P.\ Keating, 
{\it Random matrix theory and the Riemann zeros I, Three-
and four-point correlations}, Nonlinearity {\bf 8} (1995),
1115--1131; {\it II, $n$-point correlations}, ibid {\bf 9}
(1996), 911--935. [{\it sec \ref{sec:9}}] 

\bibitem{BD66} E.\ Bombieri and H.\ Davenport, {\it Small
differences between prime numbers}, Proc.\ Roy.\ Soc.\ Ser.\
A {\bf 293} (1966), 1--18. [{\it sec \ref{sec:1}}]

\bibitem{Bu07} F.\ J.\ van de Bult, {\it Counts of prime
pairs}, Report, University of Amsterdam, February 2007.
[{\it sec \ref{sec:1}}]

\bibitem{Cha03} Tsz Ho Chan, {\it More precise pair
correlation of zeros and primes in short intervals}, J.\
London Math.\ Soc.\ (2) {\bf 68} (2003), 579--598. [{\it sec
\ref{sec:9}}]

\bibitem{Cha04a} Tsz Ho Chan, {\it More precise pair
correlation conjecture on the zeros of the Riemann zeta
function}, Acta Arith.\ {\bf 114} (2004), 199--214. [{\it sec
\ref{sec:9}}]

\bibitem{Cha04b} Tsz Ho Chan, {\it Pair correlation of the
zeros of the Riemann zeta function in longer ranges}, Acta
Arith.\ {\bf 115} (2004), 181--204. [{\it sec \ref{sec:9}}]

\bibitem{Che73} Jing Run Chen, {\it On the
representation of a larger even integer as the sum of a
prime and the product of at most two primes}, Sci.\ Sinica
{\bf 16} (1973), 157--176. [{\it sec \ref{sec:1}}]

\bibitem{EH68} P.\ D.\ T.\ A.\ Elliott and H.\ Halberstam,
{\it A conjecture in prime number theory}, Symposia
Mathematica,  vol.\ 4 (INDAM, Rome, 1968/69), pp 59--72,
Academic Press, London. [{\it sec \ref{sec:9}}]

\bibitem{FG95} J.\ B.\ Friedlander and D.\ A.\ Goldston,
{\it Some singular series averages and the distribution of
Goldbach numbers in short intervals}, Illinois J.\ Math.\
{\bf 39} (1995), 158--180. [{\it sec \ref{sec:1}}]

\bibitem{Ga85} P.\ X.\ Gallagher, {\it Pair correlation of
zeros of the zeta function}, J.\ Reine Angew.\ Math.\ {\bf
362} (1985), 72--86. [{\it sec \ref{sec:9}}]

\bibitem{GaM78} P.\ X.\ Gallagher and J.\ H.\ Mueller, 
{\it Primes and zeros in short intervals}, J.\ Reine Angew.\
Math.\ {\bf 303/304} (1978), 205--220. [{\it sec
\ref{sec:9}}]

\bibitem{Go88} D.\ A.\ Goldston, {\it On the pair
correlation conjecture for zeros of the Riemann
zeta-function}, J.\ Reine  Angew.\ Math.\ {\bf 385} (1988),
24--40. [{\it sec \ref{sec:9}}]

\bibitem{Go05} D.\ A.\ Goldston, {\it Notes on pair
correlation of zeros and prime numbers}, in: Recent
Perspectives in Random Matrix Theory and Number Theory, pp
79--110, London Math.\ Soc.\ Lecture Note Ser.\, 322,
Cambridge Univ.\ Press 2005. [{\it sec \ref{sec:9}}]

\bibitem{GG90} D.\ A.\ Goldston and S.\ M.\ Gonek, {\it A
note on the number of primes in short intervals}, Proc.\
Amer.\ Math.\ Soc.\ {\bf 108} (1990), 613--620. [{\it
sec \ref{sec:9}}]

\bibitem{GG98} D.\ A.\ Goldston and S.\ M.\ Gonek, {\it Mean
value theorems for long Dirichlet polynomials and tails of
Dirichlet series}, Acta Arith.\ {\bf 84} (1998), 155--192. 
[{\it sec \ref{sec:9}}]

\bibitem{GGOS00} D.\ A.\ Goldston, S.\ M.\ Gonek,
A.\ E.\ \"{O}zl\"{u}k and C.\ Snyder, {\it On the pair
correlation of zeros of the Riemann zeta-function}, Proc.\
London Math.\ Soc.\ (3) {\bf 80} (2000), 31--49. [{\it sec
\ref{sec:1}, \ref{sec:9}}]

\bibitem{GM87} D.\ A.\ Goldston and H.\ L.\ Montgomery,
{\it Pair correlation of zeros and primes in short
intervals}, in: Analytic Number Theory and Diophantine
Problems (Stillwater, OK, 1984), Progr.\ Math.\ 70,
Birkh\"{a}user, Boston etc.\ 1987, pp 183--203. [{\it sec
\ref{sec:1},
\ref{sec:9}}]

\bibitem{GM06} D.\ A.\ Goldston, Y.\ Motohashi, J.\ Pintz and
C.\ Y.\ Yildirim, {\it Small gaps between primes exist},
Proc.\ Japan Acad.\ Ser.\ A Math.\ Sci.\ {\bf 82} (2006),
61--65. [{\it sec \ref{sec:1}}]

\bibitem{GPY05} D.\ A.\ Goldston, J.\ Pintz and
C.\ Y.\ Yildirim, {\it Primes in tuples I}, to
appear in Ann.\ of Math.\ 2007. [{\it sec \ref{sec:1},
\ref{sec:9}}]

\bibitem{HR74} H.\ Halberstam, and H.-E.\ Richert,
{\it Sieve Methods}, Academic Press, London, 1974. [{\it sec
\ref{sec:1}}]

\bibitem{HL23} G.\ H.\ Hardy and J.\ E.\ Littlewood,
{\it Some problems of `partitio numerorum' III: On the
expression of a number as a sum of primes}, Acta Math.\ {\bf
44} (1923), 1--70. [{\it sec \ref{sec:1}}]

\bibitem{HB82} D.\ R.\ Heath-Brown, {\it Gaps between
primes, and the pair correlation of zeros of the
zeta-function}, Acta Arith.\ {\bf 41} (1982), 85--99. [{\it
sec \ref{sec:9}}]

\bibitem{He94} D.\ A.\ Hejhal, {\it On the triple
correlation of zeros of the zeta function}, Internat.\ Math.\
Res.\ Notices {\bf 1994}, 293--302. [{\it sec \ref{sec:9}}]

\bibitem{Ik31} S.\ Ikehara, {\it An extension of
Landau's theorem in the analytic theory of numbers}, J.\
Math.\ and Phys.\ M.I.T.\ {\bf 10} (1931), 1--12. [{\it sec
\ref{sec:1}}]

\bibitem{Ko02} J.\ Korevaar, {\it A century of complex
Tauberian theory}, Bull.\ Amer.\ Math.\ Soc.\ (N.S.) {\bf 39}
(2002), 475--531. [{\it sec \ref{sec:1}}]

\bibitem{Ko04} J.\ Korevaar, {\it Tauberian Theory, a
Century of Developments}, Grundl.\ math.\ Wiss.\ vol.\ 329,
Springer, Berlin, 2004. [{\it sec \ref{sec:1}}]

\bibitem{Ko05} J.\ Korevaar, {\it Distributional
Wiener--Ikehara theorem and twin primes}, Indag.\ Math.\
(N.S.) {\bf 16} (2005), 37--49. [{\it sec \ref{sec:1}}]

\bibitem{Ko07} J.\ Korevaar, {\it Prime pairs and the zeta
function}. Manuscript based on two lectures,
Amsterdam, Spring 2007. [{\it sec \ref{sec:2},
\ref{sec:10}}]

\bibitem{LMS05} LMS Notes: {\it Recent Perspectives in Random
Matrix Theory and Number Theory} (F.\ Mezzadri and N.\ C.\
Snaith, eds.), London Math.\ Soc.\ Lecture Note Ser.\ vol.\
322, Cambridge Univ.\ Press, 2005. [{\it sec \ref{sec:9}}]

\bibitem{Mo73} H.\ L.\ Montgomery, {\it The pair correlation
of zeros of the zeta function}, in: Analytic Number Theory
(Proc.\ Symp.\ Pure Math., vol.\ 24), pp 181--193, Amer.\
Math.\ Soc., Providence, RI, 1973. [{\it sec \ref{sec:1},
\ref{sec:3}, \ref{sec:9}}]

\bibitem{MS04} H.\ L.\ Montgomery and K.\ Soundararajan,
{\it Primes in short intervals}, Commun.\ Math.\ Phys.\ {\bf
252} (2004), 589--617. [{\it sec \ref{sec:9}}]

\bibitem{Ni05} T.\ R.\ Nicely, {\it Counts of twin-prime
pairs and Brun's constant to $5\cdot10^{15}$}, at
http://www.trnicely.net/twins/tabpi2.html (2005).
[{\it sec \ref{sec:1}}]

\bibitem{RS94} Z.\ Rudnick and P.\ Sarnak, {\it The $n$-level
correlations of zeros of the zeta function}, C.\ R.\ Acad.\ 
Sci.\ Paris S\'{e}r.\ I Math.\ {\bf 319} (1994), 1027--1032.
[{\it sec \ref{sec:9}}]

\bibitem{RS96} Z.\ Rudnick and P.\ Sarnak, {\it Zeros of
principal $L$-functions and random matrix theory}, Duke
Math.\ J.\  {\bf 81} (1996), 269--322. [{\it sec
\ref{sec:9}}]

\bibitem{So07} K.\ Soundararajan, {\it Small gaps between
prime numbers: The work of Goldston--Pintz--Yildirim}, Bull.\
Amer.\ Math.\ Soc.\ (N.S.) {\bf 44} (2007), 1--18.
[{\it sec \ref{sec:1}}]

\bibitem{Te06} G.\ Tenenbaum, {\it The prime-pair constants
have mean value one, a simple proof}. Proposed in e-mail dated
October 29, 2006. [{\it sec \ref{sec:1}}]

\bibitem{Ti86} E.\ C.\ Titchmarsh, {\it The Theory of the
Riemann  Zeta-Function}, first edition 1951, second edition
edited by D.\ R.\ Heath-Brown, Clarendon Press, Oxford,
1986. [{\it sec \ref{sec:5}, \ref{sec:6}, \ref{sec:7}}]

\bibitem{WW27} E.\ T.\ Whittaker and G.\ N.\ Watson, {\it A
Course of Modern Analysis}, Cambridge  Univ.\ Press, 1927
(reprinted 1996). [{\it sec \ref{sec:8}}]

\bibitem{Wi32} N.\ Wiener, {\it Tauberian theorems}, Ann.\ of
Math.\ {\bf 33} (1932), 1--100. [{\it sec \ref{sec:1}}]

\bibitem{Wu04} Jie Wu, {\it Chen's double sieve, Goldbach's
conjecture and the twin prime problem}, Acta Arith.\ {\bf
114} (2004), 215--273. [{\it sec \ref{sec:1}}]


\end{thebibliography}
\end{document}